\documentclass[10pt]{article}
\usepackage{hyperref}
\usepackage{amsmath}
\usepackage[dvips]{graphicx}
\usepackage{amsthm}
\usepackage{epsfig}
\usepackage{amssymb}
\usepackage{dsfont}

\newtheorem{theorem}{Theorem}[section]

\newtheorem{pro}{Proposition}[section]

\newtheorem{proposition}[theorem]{Proposition}

\catcode`@=11
\@addtoreset{equation}{section}
\catcode`@=12

\begin{document}

\title{Decay properties of the Hardy-Littlewood-Sobolev systems
of the Lane-Emden type}
\author{Yutian Lei and Congming Li}

\date{}
\maketitle

\begin{abstract}
In this paper, we study the asymptotic behavior of positive solutions
of the nonlinear differential systems of Lane-Emden type $2k$-order equations
$$
\left \{
   \begin{array}{l}
      (-\Delta)^k u=v^q,u>0 \quad in ~R^n,\\
      (-\Delta)^k v=u^p,v>0 \quad in ~R^n,
   \end{array}
   \right.
$$
and the Hardy-Littlewood-Sobolev (HLS) type system of nonlinear equations
$$
\left \{
   \begin{array}{l}
      u(x)=\displaystyle\int_{R^n}\frac{v^q(y)dy}{|x-y|^{n-\alpha}},u>0 \quad in ~R^n,\\
      v(x)=\displaystyle\int_{R^n}\frac{u^p(y)dy}{|x-y|^{n-\alpha}},u>0 \quad in ~R^n.
   \end{array}
   \right.
$$
Such an integral system is related to the study the extremal functions of the HLS inequality.
We point out that the bounded solutions $u,v$ converge to zero either with
the fast decay rates or with the slow decay rates when $|x| \to \infty$ under some assumptions.
In addition, we also find a criterion to distinguish the fast and the slow decay
rates: if $u,v$ are the integrable solutions (i.e. $(u,v) \in L^{r_0}(R^n)
\times L^{s_0}(R^n)$), then they decay fast; if the bounded solutions $u,v$
are not the integrable solutions (i.e. $(u,v) \not\in L^{r_0}(R^n)
\times L^{s_0}(R^n)$), then they decay almost slowly. Here, for the HLS type system,
$r_0=\frac{n(pq-1)}{\alpha(q+1)}$, $s_0=\frac{n(pq-1)}{\alpha(p+1)}$; and for
the Lane-Emden type system, $r_0,s_0$ are still the forms above where $\alpha$
is replaced by $2k$.
\end{abstract}



\section{Introduction }

Let $n \geq 3$ and $p>1$. In this paper, we are concerned with the asymptotic behavior of positive solutions
of the Hardy-Littlewood-Sobolev (HLS) type system of nonlinear equations
\begin{equation} \label{hls}
\left \{
   \begin{array}{l}
      u(x)=\displaystyle\int_{R^n}\frac{v^q(y)dy}{|x-y|^{n-\alpha}},u>0 \quad in ~R^n,\\
      v(x)=\displaystyle\int_{R^n}\frac{u^p(y)dy}{|x-y|^{n-\alpha}},v>0 \quad in ~R^n.
   \end{array}
   \right.
\end{equation}
Under the assumption of the non-subcritical condition
\begin{equation} \label{ncc}
\frac{1}{p+1}+\frac{1}{q+1} \leq \frac{n-\alpha}{n},
\end{equation}
we obtain that if $u,v$ are the
integrable solutions, then they converge to $0$ with the fast decay rates
when $|x| \to \infty$. Moreover, we also point out that the equivalence relation of the
integrable solutions, the finite energy solutions, and the bounded solutions with fast decay
rates. On the other hand, we prove that the bounded
solutions decay almost slowly if the solutions are not integrable solutions.
Those decay rates are helpful to understand the existence of
positive solutions: in the supercritical case, the energy of
positive solutions is infinite and hence the variational methods
cannot use to investigate the existence. We can search for
positive solutions in the functions class whose elements decay
with the slow rates.

Recall the asymptotic behavior of the
positive solutions of the Lane-Emden equation
\begin{equation} \label{le-s}
-\Delta u=u^p, u>0 \quad in ~ R^n.
\end{equation}
(R1): When $|x| \to \infty$, then $u(x)$ converges to zero either fast
by $u(x) \simeq |x|^{2-n}$ or slowly by $u(x) \simeq |x|^{-\frac{2}{p-1}}$
(cf. \cite{YiLi}).

Here $f(x) \simeq g(x)$ means there exists $C>0$ such that
$\frac{g(x)}{C} \leq f(x) \leq Cg(x)$ when $|x| \to \infty$. Similar
results are also found in \cite{Lei} and \cite{LLD}.

We expect to generalize this result (R1) to the positive solutions
of the higher-order Lane-Emden type systems
\begin{equation} \label{2ks}
\left \{
   \begin{array}{l}
      (-\Delta)^k u=v^q, u>0 \quad in ~R^n,\\
      (-\Delta)^k v=u^p, v>0 \quad in ~R^n.
   \end{array}
   \right.
\end{equation}
Here $k\in [1,n/2)$ is an integer, $p,q>0$ and $pq>1$.
The classification of the solutions of (\ref{2ks})
has provided an important ingredient in the study of the
prescribing scalar curvature problem. The positive solutions of
(\ref{2ks}) and its corresponding single equation were studied rather
extensively (cf.\cite{CGS}, \cite{CY}, 
\cite{CL1}, \cite{GNN}, \cite{Li}, \cite{Lin}, \cite{WX}
and the references therein).

The decay rates of the positive solutions play an important
role in study the properties of the Lane-Emden
type PDEs (cf. \cite{DdMW}, \cite{LeiLi} and \cite{Z-Tran}).
Recently, Chen and Li \cite{CL3} proved the equivalence between
(\ref{2ks}) and the system involving the Riesz potentials
$$\left \{
   \begin{array}{l}
      u(x)=\displaystyle\int_{R^n}\frac{v^q(y)dy}{|x-y|^{n-2k}},u>0 \quad in ~R^n,\\
      v(x)=\displaystyle\int_{R^n}\frac{u^p(y)dy}{|x-y|^{n-2k}},v>0 \quad in ~R^n.
   \end{array}
   \right.
$$
Thus, we investigate the more general Hardy-Littlewood-Sobolev
(HLS) type integral system (\ref{hls})
where $\alpha \in (0,n)$ and $p,q>0$, $pq>1$.
The positive solutions $u,v$ of (\ref{hls}) are called the {\it
integrable solutions} if $(u,v) \in L^{r_0}(R^n) \times
L^{s_0}(R^n)$. Here
$$
r_0=\frac{n(pq-1)}{\alpha(q+1)}, \quad
s_0=\frac{n(pq-1)}{\alpha(p+1)},
$$
Moreover, if the critical condition
\begin{equation} \label{cc}
\frac{1}{p+1}+\frac{1}{q+1}=\frac{n-\alpha}{n}
\end{equation}
holds, then $r_0=p+1$ and $s_0=q+1$. The positive solutions
$(u,v) \in L^{p+1}(R^n) \times L^{q+1}(R^n)$ are called the {\it finite
energy solutions}.

The system (\ref{hls}) is related to the Euler-Lagrange system of
the extremal functions of the HLS inequality (cf. \cite{ChLO},
\cite{YLi}, \cite{Lieb}). For the finite energy solutions,
Chen, Li and Ou \cite{CLO-CPDE} proved the radial symmetry. Jin
and Li \cite{JL} obtained the optimal integrability intervals.
Hang \cite{Hang} proved the smoothness. The fast decay rates was obtained
in \cite{LLM-CV}. For the integrable solutions,
Chen and Li \cite{CL-DCDS} proved the radial symmetry. In this paper, we will establish
the integrability and the estimate the decay rates.

Recall some existence results. An important conjecture is that the HLS
type systems (\ref{hls}) has no any positive solution under the
subcritical condition:
\begin{equation} \label{succ}
\frac{1}{p+1}+\frac{1}{q+1}>\frac{n-\alpha}{n}.
\end{equation}
When $\alpha=2$, it is the well known Lane-Emden conjecture. It is still open except for
$n \leq 4$ (cf. \cite{M-DIE}, \cite{SZ-DIE}, \cite{Souplet}).
Chen and Li \cite{CL-DCDS} proved the nonexistence
of the integrable solutions of (\ref{hls}). The nonexistence of
the radial solution can be seen in \cite{CDM}
and \cite{LGZ}.

\vskip 3mm

In the critical case, we have the following result.

\begin{pro} \label{prop1.1}
(Theorem 1.2 in \cite{LeiLi}) The system (\ref{hls}) has the finite energy solutions
if and only if the critical condition (\ref{cc}) holds.
\end{pro}

So the existence was proved by Lieb who pointed out that the extremal functions
of the HLS inequality solve (\ref{hls}) (cf. \cite{Lieb}).

Proposition \ref{prop1.1} implies that the energy of the solutions is
infinite in the supercritical case. Therefore, it seems difficult
to prove the existence of positive solutions by the variational
methods.

For the scalar equation of (\ref{2ks}) with $k=1$
$$
-\Delta u =u^p,
$$
paper \cite{DdMW} shows
the existence of positive solutions with slow decay rate in the
supercritical case.
Recently, Li \cite{Li2011} proved the existence of positive solutions of (\ref{2ks}) under the
supercritical condition
$$
\frac{1}{p+1}+\frac{1}{q+1}<\frac{n-2k}{n}
$$
by means of the shooting method. Therefore, we always assume
in this paper the non-subcritical condition (\ref{ncc}) holds.

Next, we list the decay results stated by four theorems.

\begin{theorem} \label{th1.1}
Let $p\leq q$, and $u,v$ be positive solutions of (\ref{hls}).
Then there exists $c>0$ such that as $|x| \to \infty$,
$$
u(x) \geq \frac{c}{|x|^{n-\alpha}};
\quad
v(x) \geq \frac{c}{|x|^{\min\{n-\alpha,pn-(p+1)\alpha\}}}.
$$
Moreover, if $u,v$ are bounded decaying solutions, and
there exists some $\epsilon_0>0$
such that for $|y| \leq |x|$, $u(y) \geq \epsilon_0 u(x)$ or
$v(y) \geq \epsilon_0 v(x)$,
then there exists $C>0$ such that as $|x| \to \infty$,
$$
u(x) \leq C|x|^{-\frac{\alpha(q+1)}{pq-1}};
\quad
v(x) \leq C|x|^{-\frac{\alpha(p+1)}{pq-1}}.
$$
\end{theorem}

\paragraph{Remark 1.1.}
When $\alpha>2$, $u,v$ are monotonicity decreasing and hence satisfy the condition
in Theorem \ref{th1.1} (2), as long as $u,v$ are radially symmetric or bounded.
On the contrary, if the radial solutions $u,v$ are not bounded, then
(\ref{hls}) has the singular solutions $(u,v)=(C|x|^{-\frac{\alpha(q+1)}{pq-1}},
C|x|^{-\frac{\alpha(p+1)}{pq-1}})$ with some $C>0$.

\vskip 3mm

Let $p \leq q$. According to
Theorem 1.5 (2) in \cite{LeiLi}, we know that $pq>1$ and $\frac{\alpha(q+1)}{pq-1}
< n-\alpha$. This implies $\frac{\alpha(p+1)}{pq-1}<\min\{n-\alpha,pn-(p+1)\alpha\}$.
When $|x| \to \infty$,
the exponents $n-\alpha$ and $\min\{n-\alpha,pn-(p+1)\alpha\}$ of $|x|^{-1}$ are called
the fast decay rates of $u$ and $v$ respectively. The exponents $\frac{\alpha(q+1)}{pq-1}$
and $\frac{\alpha(p+1)}{pq-1}$ are called the slow ones of $u$ and $v$.

Theorem \ref{th1.1} shows that the decay rates of $u,v$ cannot be larger than the fast
rates. Moreover, if $u$ or $v$ has some monotonicity, then their decay exponents
must be between the fast and the slow rates.

The following result shows that if $u,v$ are the integrable solutions, then they
decay fast.

\begin{theorem} \label{th1.2}
Let $p\leq q$, and $(u,v)$ be a pair of positive solutions of
(\ref{hls}) with the non-subcritical condition (\ref{ncc}).
The following three items are equivalent:

(1) $(u,v) \in L^{r_0}(R^n) \times L^{s_0}(R^n)$, i.e. $u,v$ are the
integrable solutions.

(2) $u,v$ are bounded, and decay fast when $|x| \to \infty$:
$$
u(x) \simeq |x|^{\alpha-n};
$$
$$\begin{array}{lll}
&v(x)\simeq |x|^{\alpha-n} \quad &if ~p(n-\alpha)>n;\\
&v(x)\simeq |x|^{\alpha-n}\ln|x| \quad &if ~p(n-\alpha)=n;\\
&v(x)\simeq |x|^{(\alpha-n)(p+1)+n} \quad &if ~p(n-\alpha)<n.
\end{array}
$$

(3) $(u,v) \in L^{p+1}(R^n) \times L^{q+1}(R^n)$, i.e. $u,v$ are the finite
energy solutions.
\end{theorem}

\paragraph{Remark 1.2.}
It should be pointed out that the condition $p \leq q$ in Theorems
\ref{th1.1} and \ref{th1.2} is not essential.
If $q \leq p$, then the conclusions also hold as long as the positions of
$u$ and $v$ are exchanged.

\vskip 3mm

According to Proposition \ref{prop1.1}, if the supercritical condition
\begin{equation} \label{ucc}
\frac{1}{p+1}+\frac{1}{q+1}<\frac{n-\alpha}{n}
\end{equation}
holds, then the positive solutions are not finite energy solutions.
Theorem \ref{th1.2} shows that these solutions are not integrable solutions
and do not decay with the fast rates.
The following result shows that these solutions decay `almost slowly'.

\begin{theorem} \label{th1.3}
Let $u,v$ be positive bounded solutions of (\ref{hls}). Then

(1) there does {\bf not} exist $C>0$ such that either
$$
\left \{
   \begin{array}{l}
      u(x) \geq C(1+|x|)^{-\theta_3}, \quad or\\
      v(x) \geq C(1+|x|)^{-\theta_4},
   \end{array}
   \right.
$$
where $\theta_1<\frac{\alpha(q+1)}{pq-1}$, $\theta_2<\frac{\alpha(p+1)}{pq-1}$.

(2) Moreover, if $(u,v) \not\in L^{r_0}(R^n) \times L^{s_0}(R^n)$
(i.e they are not integrable solutions, particularly in the supercritical case),
then $u,v$ decay with rates not larger than the slow rates. Namely,
there does {\bf not} exist $C>0$ such that either
$$
\left \{
   \begin{array}{l}
      u(x) \leq C(1+|x|)^{-\theta_3}, \quad or\\
      v(x) \leq C(1+|x|)^{-\theta_4},
   \end{array}
   \right.
$$
where $\theta_3>\frac{\alpha(q+1)}{pq-1}$, $\theta_4>\frac{\alpha(p+1)}{pq-1}$.
\end{theorem}

\paragraph{Remark 1.3.}
\begin{enumerate}
\item The reason why we consider the bounded solutions is there exists
singular solutions $(u,v)=(C|x|^{-\frac{\alpha(q+1)}{pq-1}},
C|x|^{-\frac{\alpha(p+1)}{pq-1}})$ with some $C>0$.

\item According to Theorems \ref{th1.1}-\ref{th1.3}, we see that the solutions obtained
by the shooting method in \cite{Li2011} must decay with the slow rates.
\end{enumerate}
\vskip 3mm

So far, we only obtain the `almost slow' decay result as Theorem \ref{th1.3}.
It is still open whether the exactly slow decay result holds.
If $u,v$ are monotony like the condition in Theorem \ref{th1.1} (2),
then they decay slowly. In addition, assume the solutions are polynomially decaying
\begin{equation} \label{slow}
u(x) \simeq (1+|x|)^{-\theta_1}, \quad v(x) \simeq (1+|x|)^{-\theta_2},
\end{equation}
then the following theorem shows that
$u,v$ must decay slowly as long as
the supercritical condition (\ref{ucc}) holds.

\begin{theorem} \label{th1.4}
Let $u,v$ be bounded positive solutions of (\ref{hls}).
If there exist $\theta_1,\theta_2>0$
such that $u,v$ satisfy (\ref{slow}) as $|x|\to \infty$,
Then (\ref{ncc}) must hold, and
$$
\theta_1 \geq \frac{\alpha(q+1)}{pq-1}, \quad
\theta_2 \geq \frac{\alpha(p+1)}{pq-1}.
$$
Furthermore,

(1) if one strict inequality holds, then (\ref{cc}) must be true,
and $u,v$ are the finite energy solutions decaying fast like Theorem \ref{th1.2}.

(2) If the supercritical condition (\ref{ucc}) holds, then
the decay rates must be the slow ones:
$$
\theta_1=\frac{\alpha(q+1)}{pq-1}, \quad
\theta_2=\frac{\alpha(p+1)}{pq-1}.
$$
\end{theorem}

\paragraph{Remark 1.4.}
\begin{enumerate}
\item The fast decay rates of the finite energy solutions
were obtained in \cite{LLM-CV} and \cite{SL}, which
is coincident with Theorem \ref{th1.2}.

\item By virtue of $pq>1$, (\ref{2ks}) is equivalent to (\ref{hls})
with $\alpha=2k$ (cf. \cite{CL3}). Therefore,
Theorems \ref{th1.1}-\ref{th1.4} with $\alpha=2k$ are still true for
(\ref{2ks}).

\item If $p=q$ and $u\equiv v$, the system (\ref{hls}) is
reduced to the single equation. Therefore, Theorems \ref{th1.1}-\ref{th1.4}
with $p=q$ and $u \equiv v$ are still true. In particular, when $\alpha=2k$,
Theorems \ref{th1.1}-\ref{th1.4} still hold for the single $2k$-order PDE,
which is coincident with (R1).
\end{enumerate}

\section{Integrable solution and finite energy solution}

\begin{theorem} \label{th2.1}
Let $(u,v) \in L^{r_0}(R^n) \times L^{s_0}(R^n)$ be a pair of
positive solutions of (\ref{hls}). If $p \leq q$, then $(u,v)
\in L^{r}(R^n) \times L^{s}(R^n)$ for all
$$
\frac{1}{r} \in (0,\frac{n-\alpha}{n}), \quad
\frac{1}{s} \in (0,\min\{\frac{n-\alpha}{n},\frac{pn-(p+1)\alpha}{n}\}).
$$
\end{theorem}

\begin{proof}
Let $\frac{1}{r} \in
(\frac{\alpha(q-p)}{n(pq-1)},\frac{n-\alpha}{n})$ and $\frac{1}{s}
\in (0,1-\frac{\alpha(q-1)(p+1)}{n(pq-1)})$ satisfy
\begin{equation} \label{line}
\frac{1}{r}-\frac{1}{s}=\frac{1}{r_0}-\frac{1}{s_0}.
\end{equation}
By (\ref{line}) and the values of $r_0$ and $s_0$, we have
\begin{equation} \label{relation}
\frac{1}{r}+\frac{\alpha}{n}=\frac{q-1}{s_0}+\frac{1}{s}, \quad
\frac{1}{r}+\frac{\alpha}{n}=\frac{q-1}{s_0}+\frac{1}{s}.
\end{equation}

For $A>0$, set $u_A=u$ when $u>A$ or $|x|>A$; $u_A=0$ when $u \leq A$ and $|x| \leq A$.
Similarly, $v_A$ is the same definition.

For $g \in L^r(R^n)$ and $f \in L^s(R^n)$, define
$$
(T_1g)(x)=\int_{R^n}\frac{v_A^{q-1}(y)g(y)dy}{|x-y|^{n-\alpha}}, \quad
(T_2f)(x)=\int_{R^n}\frac{u_A^{p-1}(y)f(y)dy}{|x-y|^{n-\alpha}}.
$$
Noting (\ref{relation}), we can use the HLS inequality and the H\"older inequality
to obtain
$$\begin{array}{ll}
&\|T_1g\|_r \leq C\|v_A^{q-1}g\|_{\frac{nr}{n+r\alpha}} \leq C\|v_A\|_{s_0}^{q-1}\|g\|_s,\\
&\|T_2f\|_s \leq C\|u_A^{p-1}f\|_{\frac{ns}{n+s\alpha}} \leq C\|u_A\|_{r_0}^{p-1}\|f\|_r.
\end{array}
$$
Choosing $A$ sufficiently large such that
$$
C\|v_A\|_{s_0}^{q-1} \leq \frac{1}{4}, \quad
C\|u_A\|_{r_0}^{p-1} \leq \frac{1}{4},
$$
we see that $T=(T_1,T_2)$ is a contraction map from $L^r(R^n) \times L^s(R^n)$
to itself for all $(\frac{1}{r},\frac{1}{s}) \in I_1$, where
$$
I_1:=(\frac{\alpha(q-p)}{n(pq-1)},\frac{n-\alpha}{n})
\times (0,1-\frac{\alpha(q-1)(p+1)}{n(pq-1)}),
$$
and the norm
$$
\|T(g,f)\|_{L^r(R^n) \times L^s(R^n)}=\|T_1g\|_r+\|T_2f\|_s.
$$

In view of $(\frac{1}{r_0},\frac{1}{s_0}) \in I_1$, $T$ is also a contraction map from
$L^{r_0}(R^n) \times L^{s_0}(R^n)$ to itself.

Define
$$
G=\int_{R^n}\frac{(v-v_A)^q(y)dy}{|x-y|^{n-\alpha}}, \quad
F=\int_{R^n}\frac{(u-u_A)^p(y)dy}{|x-y|^{n-\alpha}}.
$$
Then the HLS inequality leads to $(G,F) \in L^r(R^n) \times L^s(R^n)$.

Since $(u,v)$ solves
$$
(g,f)=T(g,f)+(G,F),
$$
we can use the lifting lemma (Lemma 2.1 in \cite{JL}) to obtain
\begin{equation} \label{initial}
(u,v) \in
L^r(R^n) \times L^s(R^n), \quad \forall ~(\frac{1}{r},\frac{1}{s}) \in I_1.
\end{equation}

Next we extend the integrability domain from $I_1$ to
$$
(0,\frac{n-\alpha}{n})
\times (0,\min\{\frac{n-\alpha}{n},\frac{pn-(p+1)\alpha}{n}\}).
$$

First we claim
\begin{equation} \label{claim1}
q>\frac{q\alpha}{n}\frac{(p+1)(q-1)}{pq-1}
+\frac{\alpha}{n}.
\end{equation}
In fact, by the non-subcritical condition (\ref{ncc}) and $p \leq q$, we have
$\frac{1}{q+1}<\frac{2n-\alpha}{2n}$. This leads to
$$
1-\frac{\alpha}{qn}>\frac{q-1}{q+1}.
$$
Using again the non-subcritical condition (\ref{ncc}), we get
$$
1-(\frac{1}{p+1}+\frac{1}{q+1})>\frac{\alpha}{n}
[\frac{q-1}{q+1}+\frac{1}{q}(1-(\frac{1}{p+1}+\frac{1}{q+1}))].
$$
Multiplying by $(p+1)(q+1)$ yields
$$
pq-1>\frac{\alpha}{n}[(p+1)(q-1)+\frac{pq-1}{q}].
$$
Multiplying by $\frac{q}{pq-1}$ again, we see (\ref{claim1}).

Second we claim
\begin{equation} \label{claim2}
q-\frac{q\alpha}{n}\frac{(p+1)(q-1)}{pq-1}
-\frac{\alpha}{n} > \frac{\alpha(q-p)}{n(pq-1)}.
\end{equation}

In fact, by the non-subcritical condition (\ref{ncc}), we have
$$
\frac{q}{q-1}[1-(\frac{1}{p+1}+\frac{1}{q+1})]>\frac{\alpha}{n}.
$$
This leads to
$$\begin{array}{ll}
q>&\displaystyle\frac{\alpha}{n}\frac{q-1}{1-(\frac{1}{p+1}+\frac{1}{q+1})}=
\frac{\alpha}{n}\frac{(p+1)(q+1)(q-1)}{pq-1}\\[3mm]
&=\displaystyle\frac{\alpha}{n}
[\frac{q(p+1)(q-1)}{pq-1}+1+\frac{q-p}{pq-1}].
\end{array}
$$
This is (\ref{claim2}).

Using the HLS inequality, we have
$$
\|u\|_r \leq C\|v^q\|_{\frac{nr}{n+r\alpha}} \leq C\|v\|_{\frac{nrq}{n+r\alpha}}^q.
$$
Noting $v \in L^s(R^n)$ for all $\frac{1}{s} \in (0,1-\frac{\alpha}{n}\frac{(p+1)(q-1)}{pq-1})$
implied by (\ref{initial}),
we get $u \in L^r(R^n)$ for all
$$
\frac{1}{r} \in (0,q-\frac{q\alpha}{n}\frac{(p+1)(q-1)}{pq-1}
-\frac{\alpha}{n}).
$$
Eq. (\ref{claim1}) means this interval makes sense. Combining this with (\ref{initial}), from
(\ref{claim2}) we deduce
\begin{equation} \label{last}
u \in L^r(R^n), \quad \forall \frac{1}{r} \in (0,\frac{n-\alpha}{n}).
\end{equation}

Similarly, using the HLS inequality, we get
$$
\|v\|_s \leq C\|u\|_{\frac{nsp}{n+s\alpha}}^p.
$$
By means of (\ref{last}) we get
$$
v \in L^s(R^n), \quad \forall \frac{1}{s} \in (0,\min\{\frac{n-\alpha}{n},\frac{pn-(p+1)\alpha}{n}\}).
$$
\end{proof}

\begin{theorem} \label{th2.2}
If $u,v$ are integrable solutions of (\ref{hls}) with (\ref{ncc}),
then $u,v$ are the finite energy solutions.
\end{theorem}

\begin{proof}
First, (\ref{ncc}) implies
\begin{equation} \label{6.1}
\frac{1}{p+1},\frac{1}{q+1} \in (0,\frac{n-\alpha}{n}),
\end{equation}
and
$$
\frac{1}{p+1}+\frac{1}{q+1}<\frac{p}{(p+1)(q+1)}+\frac{n-\alpha}{n}.
$$
Thus,
$$
\frac{\alpha}{n}<\frac{p}{(p+1)(q+1)}+1-(\frac{1}{p+1}+\frac{1}{q+1})
=\frac{p+pq-1}{(p+1)(q+1)}.
$$
This result leads to
$$
\frac{1}{q+1}<p-(p+1)\frac{\alpha}{n}=\frac{pn-(p+1)\alpha}{n}.
$$
Combining this with (\ref{6.1}), and using Theorem \ref{th2.1}, we obtain
$$
(u,v) \in L^{p+1}(R^n) \times L^{q+1}(R^n).
$$
Namely, $u,v$ are the finite energy solutions.
\end{proof}

On the contrary, if (\ref{cc}) is true, we obtain $p+1=r_0$ and $q+1=s_0$.
So we also have the following result which shows that the finite energy solutions
are also the integrable solutions.

\begin{theorem} \label{th2.3}
If $(u,v) \in L^{p+1}(R^n) \times L^{q+1}(R^n)$ solves (\ref{hls}),
then $(u,v) \in L^{r_0}(R^n) \times L^{s_0}(R^n)$.
\end{theorem}

\begin{proof}
First, Proposition \ref{prop1.1} implies that (\ref{cc}) holds when
$u,v$ are the finite energy solutions. According to Theorem 1.1 in \cite{JL}, the
finite energy solutions $(u,v)$ of (\ref{hls}) with (\ref{cc}) have the following
integrability: $(u,v) \in L^r(R^n) \times L^s(R^n)$ for all $r,s$ satisfying
$$
\frac{1}{r} \in (0,\frac{n-\alpha}{n}), \quad
\frac{1}{s} \in (0,\min\{\frac{n-\alpha}{n},\frac{pn-(p+1)\alpha}{n}\}).
$$
So we only need to prove that $\frac{1}{r_0}$ and $\frac{1}{s_0}$ belong to
the corresponding intervals.

First (\ref{cc}) implies
$$
\frac{n}{p+1}<n-\alpha.
$$
Eq. (\ref{cc}) also leads to
$$
\frac{n}{p+1}=\frac{\alpha}{p+1} (1-\frac{1}{p+1}-\frac{1}{q+1})^{-1}.
$$
Combining these results yields
\begin{equation} \label{slowfast}
\frac{\alpha(q+1)}{pq-1}<n-\alpha,
\end{equation}
which means $\frac{1}{r_0} \in (0,\frac{n-\alpha}{n})$.

Noting $p \leq q$, we also have $\frac{1}{s_0} \in (0,\frac{n-\alpha}{n})$
by the same argument above. In addition, (\ref{slowfast}) leads to
$\frac{q\alpha(p+1)}{pq-1}<n$. Thus, $\frac{\alpha(p+1)}{pq-1}+(p+1)\alpha
<pn$, which implies $\frac{1}{s_0} <
\frac{pn-(p+1)\alpha}{n}$. Therefore, $\frac{1}{s_0}$ belongs
to the integrability interval.
\end{proof}

Theorems \ref{th2.2} and \ref{th2.3} show that (1) and (3) in
Theorem \ref{th1.2} are equivalent.

\section{Integrable solutions are bounded}

\begin{theorem} \label{th3.1}
If $(u,v) \in L^{r_0}(R^n) \times L^{s_0}(R^n)$ is a pair of
positive solutions of (\ref{hls}), then $u,v$ are bounded and
converge to zero when $|x| \to \infty$.
\end{theorem}

\begin{proof}
(1) Both the solutions $u$ and $v$ of (\ref{hls}) are bounded.

By exchanging the order of the integral variables, we have
$$\begin{array}{ll}
u(x) &\leq C(\displaystyle\int_0^1 \frac{\int_{B_t(x)}v^q(y)
dy}{t^{n-\alpha}} \frac{dt}{t}+\int_1^{\infty}
\frac{\int_{B_t(x)}v^q(y)
dy}{t^{n-\alpha}} \frac{dt}{t})\\[3mm]
&:=C(H_1+H_2).
\end{array}
$$
By H\"older's inequality, for any $l>1$,
$$
\int_{B_t(x)}v^q(y) dy \leq C\|v^q\|_l|B_t(x)|^{1-1/l}.
$$
Take $l$ sufficiently large such that
$\frac{1}{ql}=\varepsilon$ is sufficiently small.
According to Theorem 2.1, $\|w^q\|_l<\infty$.
Therefore,
$$
H_1 \leq C\int_0^1 \frac{|B_t(x)|^{1-q\varepsilon}}{t^{n -\alpha}}
\frac{dt}{t} \leq C\int_0^1 t^{\alpha-nq\varepsilon} \frac{dt}{t}
\leq C.
$$

If $z \in B_\delta(x)$, then $B_t(x) \subset B_{t+\delta}(z)$.
For $\delta \in (0,1)$ and $z \in B_\delta(x)$, it follows
\begin{equation} \label{tec}
\begin{array}{ll}
H_2&=\displaystyle\int_1^{\infty} \frac{\int_{B_t(x)}v^q(y)
dy}{t^{n-\alpha}} \frac{dt}{t}\\[3mm]
&\leq \displaystyle\int_1^{\infty}
\frac{\int_{B_{t+\delta}(z)}v^q(y) dy}{(t+\delta)^{n-\alpha}}
(\frac{t+\delta}{t})^{n-\alpha+1}
\frac{d(t+\delta)}{t+\delta}\\[3mm]
&\leq (1+\delta)^{n-\alpha+1}
\displaystyle\int_{1+\delta}^{\infty}
\frac{\int_{B_t(z)}v^q(y)dy}{t^{n-\alpha}} \frac{dt}{t}\leq Cu(z).
\end{array}
\end{equation}

Combining the estimates of $H_1$ and $H_2$, we have
$$
u(x) \leq C+Cu(z), \quad for \quad z \in B_\delta(x),
$$
where $\delta \in (0,1)$. Integrating on $B_\delta(x)$, we get
$$\begin{array}{ll}
u(x) &\leq C+C\displaystyle\int_{B_\delta(x)}u(z)dz \\[3mm]
&\leq
C+C\|u\|_{r_0}
|B_\delta(x)|^{1-\frac{1}{r_0}} \leq C.
\end{array}
$$
This shows $u$ is bounded in $R^n$. Similarly, $v$ is also bounded.

(2) We claim the solutions $u,v$ of (\ref{hls}) satisfy
\begin{equation}
\lim_{|x| \to \infty}u(x)=0, \quad \lim_{|x| \to \infty}v(x)=0. \label{decay}
\end{equation}

Take $x_0 \in R^n$. By (1),
$\|v\|_\infty<\infty$. Thus, $\forall \varepsilon>0$, there
exists $\delta \in (0,1)$ such that
$$
\int_0^{\delta} \frac{\int_{B_t(x_0)}v^q(z)dz}{t^{n-\alpha}}
\frac{dt}{t} \leq C\|v\|_\infty^q \int_0^{\delta} t^\alpha
\frac{dt}{t} <\varepsilon.
$$
On the other hand, similarly to the derivation of (\ref{tec}), as
$|x-x_0|<\delta$,
$$\begin{array}{ll}
&\displaystyle\int_{\delta}^{\infty}
\frac{\int_{B_t(x_0)}v^q(z)dz}{t^{n-\alpha}}
\frac{dt}{t}\\[3mm]
& \leq \displaystyle\int_{\delta}^{\infty}
\frac{\int_{B_{t+\delta}(x)}v^q(z)dz}{(t+\delta)^{n-\alpha}}
(\frac{t+\delta}{t})^{n-\alpha+1} \frac{d(t+\delta)}{t+\delta}\\[3mm]
& \leq C\displaystyle\int_0^{\infty}
\frac{\int_{B_t(x)}v^q(z)dz}{t^{n-\alpha}} \frac{dt}{t}=Cu(x).
\end{array}
$$
Combining these estimates, we get
$$
u(x_0)<\varepsilon+Cu(x), \quad for \quad |x-x_0|
<\delta.
$$
Since $u \in
L^{r_0}(R^n)$, there holds $\lim_{|x_0| \to
\infty}\int_{B_\delta(x_0)}u^{r_0}(x)dx=0$. Thus, we have
$$\begin{array}{ll}
u^{r_0}(x_0)
&=|B_\delta(x_0)|^{-1}\displaystyle\int_{B_\delta(x_0)}
u^{r_0}(x_0)dx\\[3mm] &\leq
C\varepsilon^{r_0}
+C|B_\delta(x_0)|^{-1}\displaystyle\int_{B_\delta(x_0)}u^{r_0}(x)dx
\to 0
\end{array}
$$
when $|x_0| \to \infty$ and $\varepsilon \to 0$.
Similarly, $v$ has the same result. Thus, (\ref{decay}) holds.
\end{proof}

\section{Fast decay of integrable solutions}

In this section, we always assume that $(u,v)$ is a pair of positive solutions
of the system (\ref{hls}) with (\ref{ncc}).

First we verify the integrable solutions decay fast. This argument includes
five propositions.

\begin{proposition}  \label{prop4.1}
$B_0:=\int_{R^n} v^q(y)dy  <\infty$.
\end{proposition}

\begin{proof}
By (\ref{ncc}) and $q \geq p$, we have
$$
\frac{1}{q} \leq \frac{n-\alpha}{n+\alpha}<\frac{n-\alpha}{n}.
$$
On the other hand, in view of
$\frac{1}{p+1}<1-\frac{1}{q+1}=\frac{q}{q+1}$, it follows from
(\ref{ncc}) that
\begin{equation} \label{jia}
\frac{q+1}{q(p+1)}=\frac{1}{p+1}+\frac{1}{q(p+1)}
<\frac{1}{p+1}+\frac{1}{q+1}<\frac{n-\alpha}{n}.
\end{equation}
Therefore, $1+\frac{1}{q}<(p+1)\frac{n-\alpha}{n}
=1+\frac{pn-(p+1)\alpha}{n}$. This implies
$$
\frac{1}{q} < \frac{pn-(p+1)\alpha}{n}.
$$
According to Theorem \ref{th2.1}, $v \in L^q(R^n)$.
\end{proof}

\begin{proposition} \label{prop4.2}
$$
\lim_{|x| \to \infty} u(x)|x|^{n-\alpha}=B_0.
$$
\end{proposition}

\begin{proof}
For fixed $R>0$, write
$$
L_1:=\int_{B_R} v^q(y)
(\frac{|x|^{n-\alpha}}{|x-y|^{n-\alpha}}-1)dy.
$$
When $y \in B_R$ and $|x| \to \infty$,
$$
v^q(y)
|\frac{|x|^{n-\alpha}}{|x-y|^{n-\alpha}}-1| \leq
3 v^q(y)  \in L^1(R^n)
$$
by virtue of Proposition \ref{prop4.1}. Using Lebesgue's dominated convergence
theorem yields
$$
|L_1| \to 0,\quad as \quad |x| \to \infty.
$$
This result leads to
$$
\lim_{R \to \infty}\lim_{|x| \to \infty}
\int_{B_R} v^q(y)
\frac{|x|^{n-\alpha}}{|x-y|^{n-\alpha}}dy=B_0.
$$

Next, we write
$$
L_2:=\int_{(R^n\setminus B_R)\setminus B(x,|x|/2)} v^q(y)
\frac{|x|^{n-\alpha}}{|x-y|^{n-\alpha}}dy.
$$
Clearly, $|x-y| \geq |x|/2$ when $y \in
(R^n\setminus B_R)\setminus B(x,|x|/2)$.
Therefore, when $R \to \infty$,
$$
L_2 \leq C\int_{R^n\setminus
B_R} v^q(y)dy \to 0.
$$

We write
$$
L_3:=\int_{B(x,|x|/2)} v^q(y)
\frac{|x|^{n-\alpha}}{|x-y|^{n-\alpha}}dy
$$
and
$$
L_4:=\frac{L_3}{|x|^{n-\alpha}}
=\int_{B(x,|x|/2)}\frac{v^q(y)dy}{|x-y|^{n-\alpha}}.
$$
According to Theorem 3.1 in \cite{CL-DCDS}, $v$ is radially symmetric and
decreasing about $x_0 \in R^n$. Without loss of generality,
we can view $x_0$ as the origin when $|x|$ is sufficiently large.
Therefore,
$$
L_4 \leq
v^q(x/2) \int_{B(x,|x|/2)}\frac{dy}{
|x-y|^{n-\alpha}} \leq \frac{Cv^q(x/2)}{|x|^{-\alpha}}.
$$
Write $r=|x|$, and define $\tilde{v}(r)=v(x)$. Thus,
\begin{equation}
L_4 \leq C\tilde{v}^q(|x|/2)|x|^\alpha
\label{3.4}
\end{equation}

On the other hand, Theorem \ref{th2.1} shows that
$$
v \in L^s(R^n), \quad \frac{1}{s}=\frac{1}{n} \min\{n-\alpha,
p(n-\alpha)-\alpha\}-\frac{\varepsilon}{n}.
$$
Here $\varepsilon>0$ is sufficiently small. This integrability
result, together with the decreasing property of $v$, implies
\begin{equation}
\tilde{v}^s(|x|/2)|x|^n \leq C\int_{B(0,\frac{|x|}{2}) \setminus
B(0,\frac{|x|}{4})}v^s(y)dy \leq C. \label{3.5}
\end{equation}

We claim
\begin{equation}
|x|^{n-\alpha}L_4=o(1), \quad as \quad |x|\to \infty.
\label{3.6}
\end{equation}

We will prove (\ref{3.6}) in two cases.

Case 1. When $n-\alpha \leq
p(n-\alpha)-\alpha$, (\ref{3.5}) means
$$
\tilde{v}^q(|x|/2)|x|^{q(n-\alpha-\varepsilon)} \leq C.
$$
Combining this consequence with (\ref{3.4}) yields
$$
|x|^{q(n-\alpha-\varepsilon)-\alpha}L_4 \leq
C\tilde{v}^q(|x|/2)|x|^{q(n-\alpha-\varepsilon)} \leq C.
$$
In view of $q(n-\alpha)>n$, we have
$q(n-\alpha)-\alpha>n-\alpha$. Choosing
$\varepsilon$ properly small and letting $|x| \to \infty$ in the
result above, we can deduce (\ref{3.6}).

Case 2. When $n-\alpha> p(n-\alpha)-\alpha$,
(\ref{3.5}) implies
$$
\tilde{v}^q(|x|/2)|x|^{q(p(n-\alpha)-\alpha-\varepsilon)}
\leq C .
$$
Combining with (\ref{3.4}) yields
$$
|x|^{q(p(n-\alpha)-\alpha-\varepsilon)-\alpha}L_4
\leq C.
$$
By (\ref{jia}), it follows $q(p+1)(n-\alpha)>n+qn$. This is
equivalent to $q[p(n-\alpha)-\alpha]>n$. Choosing $\varepsilon$
properly small and letting $|x| \to \infty$ in the result above,
we can also obtain (\ref{3.6}).

Inserting (\ref{3.6}) into $L_3$, we derive that
$$
L_3 \to 0, \quad as \quad |x| \to \infty.
$$

Combining all the estimates of $L_1, L_2$ and $L_3$, we complete
the proof.
\end{proof}

\vskip 5mm

According to Proposition \ref{prop4.2},
there exists a properly large constant $R>0$ such that
\begin{equation}
u(x)=\frac{B_0+o(1)}{|x|^{n-\alpha}}, \quad for \quad x \in
R^n\setminus B_R. \label{3.8}
\end{equation}
Hereafter, we will use (\ref{3.8}) to investigate the decay rate of $v$.

\begin{proposition} \label{prop4.3}
If $p(n-\alpha)>n$, then $B_1=\int_{R^n} {u^p(y)dy} < \infty$. In
addition,
$$
\lim_{|x| \to \infty}v(x)|x|^{n-\alpha}= B_1.
$$
\end{proposition}

\begin{proof}
According to Theorem \ref{th2.1}, by $p(n-\alpha)>n$ we can also prove
$B_1<\infty$.

When $y \in B_R$ for large $R>0$, $\lim_{|x| \to
\infty}\frac{|x|}{|x-y|}=1$. By $B_1<\infty$ and the Lebesgue
dominated convergence theorem, it follows that
\begin{equation}
\lim_{R \to \infty}\lim_{|x| \to \infty}
\int_{B_R}\frac{|x|^{n-\alpha}u^p(y)dy}{|x-y|^{n-\alpha}}
=B_1. \label{3.9}
\end{equation}

In view of $p(n-\alpha)>n$, it is not
difficult to deduce that
\begin{equation}
\int_{R^n\setminus
B_R}\frac{|x|^{n-\alpha}dy}{|x-y|^{n-\alpha}|y|^{p(n-\alpha)}}
=o(1),\quad as \quad |x| \to \infty, R \to \infty. \label{3.10}
\end{equation}

By virtue of (\ref{3.8}), we have
$$
|x|^{n-\alpha}v(x)
=\int_{B_R}\frac{|x|^{n-\alpha}u^p(y)dy}{|x-y|^{n-\alpha}}
+\int_{R^n\setminus B_R}
\frac{(B_0+o(1))^p|x|^{n-\alpha}dy}{|x-y|^{n-\alpha}|y|^{p(n-\alpha)}}.
$$
Inserting (\ref{3.9}) and (\ref{3.10}) into the identity above, we get
$$
|x|^{n-\alpha}v(x) \to B_1, \quad as \quad |x| \to \infty.
$$
Proposition \ref{prop4.3} is proved.
\end{proof}

\begin{proposition} \label{prop4.4}
If $p(n-\alpha)=n$, then
$$
\lim_{|x| \to \infty}|x|^{n-\alpha}(\ln|x|)^{-1}v(x)=
B_0^p|S^{n-1}|.
$$
\end{proposition}

\begin{proof}
By virtue of (\ref{3.8}), for large $R>0$, we have
\begin{equation} \label{3.11}
\begin{array}{ll}
\displaystyle\frac{|x|^{n-\alpha}}{\ln|x|}v(x)=
&\displaystyle\frac{1}{\ln|x|}\int_{B_R}\frac{|x|^{n-\alpha}u^p(y)dy}{|x-y|^{n-\alpha}}\\[3mm]
&+\displaystyle\frac{(B_0+o(1))^p}{\ln|x|}\int_{R^n\setminus B_R}
\frac{|x|^{n-\alpha}dy}{|x-y|^{n-\alpha}|y|^{n}}.
\end{array}
\end{equation}

Eq. (\ref{3.9}) implies that as $|x| \to \infty$,
\begin{equation}
\frac{1}{\ln|x|}\int_{B_R}\frac{|x|^{n-\alpha}u^p(y)dy}{|x-y|^{n-\alpha}}
=o(1). \label{3.12}
\end{equation}

On the other hand, for the large constant $R>0$ and the small
constant $\delta \in (0,1/2)$,
\begin{equation} \label{3.13}
\begin{array}{ll}
\displaystyle\frac{1}{\ln|x|}\int_{R^n\setminus B_R}
\frac{|x|^{n-\alpha}dy}{|x-y|^{n-\alpha}|y|^n}
=&\displaystyle\frac{1}{\ln|x|}\int_{\frac{R}{|x|}}^{\delta}\int_{S^{n-1}}
\frac{dr}{r|e-rw|^{n-\alpha}}\\[3mm]
&+\displaystyle\frac{1}{\ln|x|}\int_{R^n\setminus
B_\delta } \frac{dz}{|z|^n|e-z|^{n-\alpha}}.
\end{array}
\end{equation}

Indeed, the second term of the right hand side is finite since
$n-\alpha<n$ near $e$; and $n-\alpha+n>n$ near infinity. Moreover,
the upper bound only depends on $\delta$. Letting $|x| \to
\infty$, we have
\begin{equation}
\frac{1}{\ln|x|}\int_{R^n\setminus B_\delta }
\frac{dz}{|z|^n|e-z|^{n-\alpha}}=o(1). \label{3.14}
\end{equation}

When $r \in (0,\delta)$, $1-\delta \leq |e-rw| \leq 1+\delta$.
There exists $\theta \in (-1,1)$ such that
$|e-rw|=1+\theta\delta$. Thus, the first term of (\ref{3.13})
$$\begin{array}{ll}
\displaystyle\frac{1}{\ln|x|}\int_{\frac{R}{|x|}}^{\delta}
\int_{S^{n-1}}\frac{dr}{r|e-rw|^{n-\alpha}}
&=\displaystyle\frac{|S^{n-1}|}{(1+\theta\delta)^{n-\alpha}\ln|x|}(\ln\delta-\ln
R+\ln|x|)\\[3mm]
& \to \displaystyle\frac{|S^{n-1}|}{(1+\theta\delta)^{n-\alpha}}
\quad (|x| \to
\infty)\\[3mm]
&\to |S^{n-1}| \quad (\delta \to 0).
\end{array}
$$

Substituting this result and (\ref{3.14}) into (\ref{3.13}), we have
$$
\frac{1}{\ln|x|}\int_{R^n\setminus B_R}
\frac{|x|^{n-\alpha}dy}{|x-y|^{n-\alpha}|y|^n} \to |S^{n-1}|, \quad
as \quad |x| \to \infty.
$$
Combining with (\ref{3.11}) and (\ref{3.12}) we can complete Proposition \ref{prop4.4}.
\end{proof}

\begin{proposition} \label{prop4.5}
If $p(n-\alpha)<n$, then
$$
B_3:=B_0^p\int_{R^n}|z|^{-(n-\alpha)p}|e-z|^{-n+\alpha}dz
< \infty.
$$
In addition,
$$
\lim_{|x| \to \infty}|x|^{pn-\alpha(p+1)}v(x)=
B_3.
$$
\end{proposition}

\begin{proof}
It is easy to see $B_3<\infty$, since we observe that
$\frac{1}{p+1}<\frac{n-\alpha}{n}$ means that the integral
decays at the rate $(n-\alpha)(p+1)>n$ near infinite,
$n-\alpha<n$ near $e$, and $p(n-\alpha)<n$ near the origin.

For large $R>0$, using
(\ref{3.8}) we have
\begin{equation}
\begin{array}{ll}
|x|^{pn-\alpha(p+1)}v(x)=&|x|^{p(n-\alpha)-n}
\displaystyle\int_{B_R}\frac{|x|^{n-\alpha}u^p(y)dy}{|x-y|^{n-\alpha}}\\[3mm]
&+(B_0+o(1))^p\displaystyle\int_{R^n \setminus B_R}
\frac{|x|^{(p+1)(n-\alpha)-n}dy}{|x-y|^{n-\alpha}|y|^{p(n-\alpha)}}.
\end{array}
\label{3.15}
\end{equation}

When $y \in B_R$, and $|x| \to \infty$,
\begin{equation}
|x|^{p(n-\alpha)-n}
\int_{B_R}\frac{|x|^{n-\alpha}u^p(y)dy}{|x-y|^{n-\alpha}}
\leq C |x|^{p(n-\alpha)-n} \to 0, \label{3.16}
\end{equation}
since $p(n-\alpha)<n$.

On the other hand, when $|x| \to \infty$,
$$
\int_{R^n \setminus B_R}
\frac{|x|^{(p+1)(n-\alpha)-n}dy}{|x-y|^{n-\alpha}|y|^{p(n-\alpha)}}
=\int_{R^n \setminus B_{R/|x|}}
\frac{dz}{|z|^{p(n-\alpha)}|e-z|^{n-\alpha}} \to
\frac{B_3}{B_0^p}.
$$
Inserting this result and (\ref{3.16}) into (\ref{3.15}), we complete the
proof of Proposition.
\end{proof}

Next, we verify that the bounded solutions with fast decay rates
must be the integrable solutions.

\begin{proposition} \label{prop4.6}
Let $u,v$ solve (\ref{hls}) with (\ref{ncc}). If they are bounded and decay fast, then
$(u,v) \in L^{r_0}(R^n) \times L^{s_0}(R^n)$.
\end{proposition}

\begin{proof}
From (\ref{ncc}), we can see easily that $p>\frac{\alpha}{n-\alpha}$.
This results together with (\ref{ncc}) lead to
$$
\frac{\alpha}{(n-\alpha)(p+1)}<1-(\frac{1}{p+1}+\frac{1}{q+1}).
$$
Multiplying by $(p+1)(q+1)$ yields $\frac{\alpha(q+1)}{n-\alpha}<pq-1$.
This implies $n<(n-\alpha)r_0$. Since $u$ is bounded and decay fast,
there holds
$$
\int_{R^n} u^{r_0}(x)dx \leq C+C\int_R^\infty r^{n-(n-\alpha)r_0} \frac{dr}{r}<\infty.
$$
Namely, $u \in L^{r_0}(R^n)$.

Noting $p\leq q$, we also have $n-(n-\alpha)s_0<0$.
if $v$ is bounded and decaying with the rate $|x|^{\alpha-n}$,
we also deduce $v \in L^{s_0}$ by the same argument above.

If $v$ is bounded and decaying with the rate $|x|^{\alpha-n}\ln|x|$, then
there exists a suitably large $R>0$ such that $(\ln|x|)^{s_0} \leq |x|^\epsilon$
for $|x|>R$, where $\epsilon>0$ is sufficiently small. Then, by
$n-(n-\alpha)s_0<0$, we also get
$$
\int_{R^n} v^{s_0}(x)dx \leq C+C\int_R^\infty r^{n-(n-\alpha)s_0+\epsilon}
\frac{dr}{r}<\infty.
$$

Let $v$ be bounded and decaying with the rate $|x|^{(\alpha-n)(p+1)+n}$. From
(\ref{ncc}) we have $\frac{1}{p+1}<\frac{n-\alpha}{n}$. This and (\ref{ncc})
lead to $\frac{\alpha(q+1)}{pq-1}<n-\alpha$, which implies
$\frac{pn(pq-1)}{\alpha(p+1)}>pq$. From this we deduce that
$n-[pn-\alpha(p+1)]s_0<0$, and hence
$$
\int_{R^n} v^{s_0}dx \leq C+\int_R^\infty r^{n-[(n-\alpha)(p+1)+n]s_0} \frac{dr}{r}
<\infty.
$$
This means $v \in L^{s_0}(R^n)$.
\end{proof}

The argument in Sections 3 and 4 shows that (1) and (2)
in Theorem \ref{th1.2} are equivalent. Combining with the argument
in Section 2, we complete the proof of Theorem \ref{th1.2}.

In addition, we can also prove directly the following result.

\begin{proposition} \label{prop4.7}
Items (2) and (3) in Theorem \ref{th1.2} are also equivalent.
\end{proposition}

\begin{proof}
(3)$\Rightarrow$(2): First, according to Proposition \ref{prop1.1},
we see the critical condition (\ref{cc}) holds. By Theorems 1.1 and 1.3 in \cite{JL},
we also obtain the optimal integrability of the finite energy
solutions $u,v$. Based on this result, \cite{CL2} proved the
boundedness of $u$ and $v$. In addition, Theorem 2 in \cite{LLM-CV}
shows the fast decay rates of $u,v$ as Theorem \ref{th1.2} (see also
Corollary 1.3 (2) in \cite{SL}).

(2)$\Rightarrow$(3): Eq. (\ref{ncc}) leads to $n<(p+1)(n-\alpha)$. Hence,
from the boundedness and the fast decay rate of $u$, we have
$$
\int_{R^n}u^{p+1}(x)dx \leq C+C\int_R^\infty r^{n-(p+1)(n-\alpha)}\frac{dr}{r}
<\infty.
$$

Similarly, (\ref{ncc}) also leads to $n<(q+1)(n-\alpha)$. We also deduce
that $v \in L^{q+1}(R^n)$ when $p(n-\alpha) \geq n$.

Eq. (\ref{ncc}) implies
$$
\alpha<(1-\frac{1}{p+1}-\frac{1}{q+1})n+\frac{pn}{(p+1)(q+1)}.
$$
Multiplying by $(p+1)(q+1)$ yields
$$
\frac{1}{q+1}<\frac{pn-(p+1)\alpha}{n}.
$$
Thus,
$$
\int_{R^n} v^{q+1}(x)dx \leq C+C\int_R^\infty r^{n-(q+1)(pn-(p+1)\alpha)}
\frac{dr}{r}<\infty.
$$
So, $v \in L^{q+1}(R^n)$ when $p(n-\alpha)<n$.
\end{proof}

\section{Slow decay of bounded solutions}

\begin{proposition} \label{prop5.1}
Let $u,v$ be positive bounded solutions. Then there exists $c>0$
such that as $|x| \to \infty$,
\begin{equation} \label{ulb}
u(x) \geq \frac{c}{(1+|x|)^{n-\alpha}};
\end{equation}
\begin{equation} \label{vlb}
v(x) \geq \frac{c}{(1+|x|)^{\min\{n-\alpha,pn-(p+1)\alpha\}}}.
\end{equation}
\end{proposition}

\begin{proof}
First, we can find $c>0$ such that $u(y),v(y) \geq c>0$ for $y
\in B_1(0)$. Therefore,
$$
u(x) \geq c\int_{B_1(0)}\frac{dy}{|x-y|^{n-\alpha}} \geq c(1+|x|)^{\alpha-n}.
$$
This is (\ref{ulb}). Similarly, we also have
\begin{equation} \label{vlb1}
v(x) \geq \frac{c}{(1+|x|)^{n-\alpha}}.
\end{equation}
Substituting (\ref{ulb}) into $v(x) \geq \int_{B(x,|x|/2)}u^p(y)|x-y|^{\alpha-n}dy$
yields
$$
v(x) \geq c(1+|x|)^{p(\alpha-n)}\int_0^{|x|/2}r^\alpha \frac{dr}{r}
=c(1+|x|)^{(p+1)\alpha-pn}.
$$
Combining with (\ref{vlb1}), we obtain (\ref{vlb}).
\end{proof}

Theorem \ref{th1.2} shows that $u,v$ decay by the fast rates as long as
they are the integrable solutions. If $u,v$ are not integrable, we conjecture that they
decay slowly.

The following result shows that the decay rates of $u,v$ are not faster
than the slow rates $\frac{\alpha(q+1)}{pq-1}$ and $\frac{\alpha(p+1)}{pq-1}$
respectively, if $u,v$ are not integrable.

\begin{proposition} \label{prop5.2}
Let $u,v$ be positive bounded solutions, and
$\theta_3>\frac{\alpha(q+1)}{pq-1}$, $\theta_4>\frac{\alpha(p+1)}{pq-1}$.
If $u,v$ are not the integrable solutions, then there does not exist $C>0$
such that either
$$
u(x) \leq C(1+|x|)^{-\theta_3},
\quad or \quad
v(x) \leq C(1+|x|)^{-\theta_4}.
$$
\end{proposition}

\begin{proof}
Suppose there exists $C>0$ such that as $|x| \to \infty$,
$$
u(x) \leq C(1+|x|)^{-\theta_3},
$$
where $\theta_3>\frac{\alpha(q+1)}{pq-1}$,
then
$$\begin{array}{ll}
\displaystyle\int_{R^n} u^{r_0}(x)dx=&\displaystyle\int_{B_R(0)}u^{r_0}(x)dx
+\int_{R^n \setminus B_R(0)} u^{r_0}(x)dx\\[3mm]
&\leq C+C\displaystyle\int_R^\infty r^{n-r_0\theta_3} \frac{dr}{r}
< \infty. \end{array}
$$

Similarly, if $v(x) \leq C(1+|x|)^{-\theta_4}$
with $\theta_4>\frac{\alpha(p+1)}{pq-1}$,
then it also belongs to $L^{s_0}(R^n)$.

Thus, $u$ (or $v$) is integrable solution, which
contradicts with the assumption of our proposition.
\end{proof}

The following result shows that the decay rates of $u,v$ are not slower
than the slow rates $\frac{\alpha(q+1)}{pq-1}$ and $\frac{\alpha(p+1)}{pq-1}$,
respectively.

\begin{proposition} \label{prop5.3}
Let $u,v$ be positive bounded solutions of (\ref{hls}),
and $\theta_1<\frac{\alpha(q+1)}{pq-1}$, $\theta_2<\frac{\alpha(p+1)}{pq-1}$.
Then there does not exist $C>0$
such that either
$$
u(x) \geq C(1+|x|)^{-\theta_1}, \quad or \quad v(x) \geq C(1+|x|)^{-\theta_2}.
$$
\end{proposition}

\begin{proof}
If there exists $C>0$ such that for large $|x|$,
$$
u(x) \geq C(1+|x|)^{-\theta_1}, \quad \theta_1<\frac{\alpha(q+1)}{pq-1}.
$$
By an iteration we can deduce the contradiction.

Denoting $\theta_1$ by $b_0$, we have
$$
v(x) \geq \int_{B(x,|x|/2)}\frac{u^p(y)dy}{|x-y|^{n-\alpha}}
\geq c(1+|x|)^{-a_1}, \quad a_1=pb_0-\alpha.
$$
Using this result, we have
$$
u(x) \geq \int_{B(x,|x|/2)}\frac{v^q(y)dy}{|x-y|^{n-\alpha}}
\geq c(1+|x|)^{-b_1}, \quad b_1=qa_1-\alpha.
$$
By induction, we have two sequences
$$
b_0=\theta_1, \quad a_j=pb_{j-1}-\alpha,\quad
b_j=qa_j-\alpha, \quad j=1,2,\cdots.
$$
We claim that there must be $j_0$ such that $b_{j_0}<0$, which leads to
$u(x)=\infty$ for large $|x|$. In fact,
$$
b_j=pqb_{j-1}-\alpha(q+1)=\cdots=
(pq)^jb_0-\alpha(q+1)(1+pq+\cdots+(pq)^{j-1}).
$$
In view of $pq>1$, we have
$$
b_j=(pq)^j(b_0-\frac{\alpha(q+1)}{pq-1})+\frac{\alpha(q+1)}{pq-1}.
$$
Noting $b_0-\frac{\alpha(q+1)}{pq-1}<0$, we can find a large $j_0$ such
that $b_{j_0}<0$. It is impossible since the solution $u$ blows up.

Similarly, if there exists $C>0$ such that for large $|x|$,
$$
v(x) \geq C(1+|x|)^{-\theta_2}, \quad \theta_2<\frac{\alpha(p+1)}{pq-1}.
$$
By an analogous iteration argument above, we can also deduce a contradiction.
\end{proof}

Combining Propositions \ref{prop5.2} and \ref{prop5.3}, we complete
the proof of Theorem \ref{th1.3}.

\paragraph{Remark 5.1.}
Proposition \ref{prop5.3} shows that if there exists $C>0$ such that
$$
u(x) \leq C(1+|x|)^{-\theta_1}; \quad v(x) \leq C(1+|x|)^{-\theta_2},
$$
then $\theta_1 \geq \frac{\alpha(q+1)}{pq-1}$,
$\theta_2 \geq \frac{\alpha(p+1)}{pq-1}$.
However, there may be $C_j \to \infty$ such that as some $|x_j| \to \infty$,
$$
u(x_j) \leq C_j(1+|x_j|)^{-\theta_1}; \quad v(x_j) \leq C_j(1+|x_j|)^{-\theta_2}.
$$

\vskip 3mm

If $u,v$ have some monotonicity, then the result above does not happen.

\begin{proposition} \label{prop5.4}
Let $u,v$ be positive bounded decaying solutions. If there exists some $\epsilon_0>0$
such that for $|y| \leq |x|$,
$$
u(y) \geq \epsilon_0 u(x), \quad or \quad v(y) \geq \epsilon_0 v(x),
$$
then there exists $C>0$ such that
$$
u(x) \leq C(1+|x|)^{-\frac{\alpha(q+1)}{pq-1}}, \quad
v(x) \leq C(1+|x|)^{-\frac{\alpha(p+1)}{pq-1}}.
$$
\end{proposition}

\begin{proof}
Clearly,
\begin{equation} \label{ulb2}
u(x) \geq cv^q(x) \int_R^{|x|}r^\alpha \frac{dr}{r}
\geq cv^q(x)|x|^\alpha, \quad for ~lagre ~|x|.
\end{equation}
By the monotonicity of $u$ we also deduce the monotonicity of $v$.
Thus, we also have
$$
v(x) \geq cu^p(x)|x|^\alpha.
$$
Inserting this result into (\ref{ulb2}), we get
$$
v(x) \geq cv^{pq}(x)|x|^{(p+1)\alpha},
$$
which implies the estimate of $v$. Similarly, $u$ also has the
corresponding estimate.
\end{proof}

Combining Propositions \ref{prop5.1} and \ref{prop5.4},
we complete the proof of Theorem \ref{th1.1}.

Next, we prove Theorem \ref{th1.4}. It is the corollary of the following
proposition.

\begin{proposition} \label{prop5.5}
Suppose that the positive bounded solutions $u,v$ satisfy
\begin{equation} \label{assum}
u(x) \simeq (1+|x|)^{-\theta_1}, v(x) \simeq (1+|x|)^{-\theta_2},
\quad when ~|x| \to \infty.
\end{equation}
Then (\ref{ncc}) must hold, and
\begin{equation} \label{jkl}
\theta_1 \geq \frac{\alpha(q+1)}{pq-1}, \quad
\theta_2 \geq \frac{\alpha(p+1)}{pq-1}.
\end{equation}
Furthermore,

(1) if one strict inequality of (\ref{jkl}) holds, then (\ref{cc}) must be true
and $u,v$ are the finite energy solutions decaying fast like Theorem \ref{th1.2}.

(2) If $u,v$ are not the integrable solutions, then
$$
\theta_1=\frac{\alpha(q+1)}{pq-1}, \quad
\theta_2=\frac{\alpha(p+1)}{pq-1}.
$$
\end{proposition}

\begin{proof}
{\it Step 1.}
We first claim
$$
\theta_1 \geq \frac{\alpha(q+1)}{pq-1}, \quad
\theta_2 \geq \frac{\alpha(p+1)}{pq-1}.
$$
In fact, $|x|/2 \leq |y| \leq 3|x|/2$ when $y \in B_{|x|/2}(x)$.
Thus, for large $|x|$, from (\ref{assum}) it follows that
$$\begin{array}{ll}
C(1+|x|)^{-\theta_1}
&\geq u(x) \\[3mm]
&\geq \displaystyle\int_{B(x,|x|/2)} \frac{v^q(y)dy}{|x-y|^{n-\alpha}}\\[3mm]
&\geq c(1+|x|)^{-q\theta_2} \displaystyle\int_0^{|x|/2}
r^{\alpha} \frac{dr}{r}\\[3mm]
&\geq c(1+|x|)^{\alpha-q\theta_2}.
\end{array}
$$
This result implies
$$
\theta_1 \leq q\theta_2-\alpha
$$
since $|x|$ is sufficiently large. Similarly,
$$
\theta_2 \leq p\theta_1-\alpha.
$$
These two inequalities above show our claim.

{\it Step2.}
We claim that the subcritical condition (\ref{succ}) is not true. Otherwise,
$$
\theta_2 \geq \frac{\alpha(p+1)}{pq-1}=\frac{\alpha}{q+1}
(1-\frac{1}{p+1}-\frac{1}{q+1})^{-1}>\frac{n}{q+1},
$$
which implies that $v \in L^{q+1}(R^n)$ is a finite energy solution.
This contradicts with Proposition \ref{prop1.1}.

{\it Step 3.} We prove (1) and (2).

(1) Without loss of generality, we assume $\theta_1>\frac{\alpha(q+1)}{pq-1}$.
Then using Proposition \ref{prop5.2}, we know $u \in L^{r_0}(R^n)$ and hence
$u$ is the integrable solution.
By the HLS inequality, $v$ is also the integrable solution.
According to Theorem \ref{th1.2}, (\ref{cc}) is true, and $u,v$ are the finite
energy solutions.

(2) Using Proposition \ref{prop5.2}, from (\ref{assum}) and
(\ref{jkl}) we can see our conclusion.
\end{proof}

\paragraph{Acknowledgements.} The work of Y. Lei is partially supported by NSFC grant 11171158,
the Natural Science Foundation of Jiangsu (BK2012846) and SRF for ROCS, SEM.
The work of C. Li is partially supported by NSF grant DMS-0908097 and NSFC grant 11271166.

Yutian Lei

Institute of Mathematics, School of Mathematical Sciences,
Nanjing Normal University, Nanjing, 210023, China

Congming Li

Department of Applied Mathematical, University of Colorado at Boulder,
Boulder, CO 80309, USA

Department of Mathematics, and MOE-LSC, Shanghai Jiao Tong University,
Shanghai, 200240, China

\end{document}